\documentclass{article}
\usepackage[english]{babel}
\usepackage{amsmath,amssymb,latexsym,alltt}

\usepackage[colorlinks=true,linkcolor=blue,citecolor=blue,urlcolor=black,pdfpagelabels=false]{hyperref}

\newcommand{\email}[1]{{\emph{Email:} \texttt{#1}}}

\usepackage{amsthm,thmtools}
\theoremstyle{plain}
  \declaretheorem[numberwithin=section]{theorem}
  \declaretheorem[numberlike=theorem]{corollary}
  
  \declaretheorem[numberlike=theorem]{lemma}

\theoremstyle{definition}

  \declaretheorem[numberlike=theorem]{remark}

\numberwithin{equation}{section}

\newcommand{\id}{d}
\newcommand{\Li}{\operatorname{Li}}
\newcommand{\pFq}[5]{\ensuremath{{}_{#1}F_{#2} \left( \genfrac{}{}{0pt}{}{#3}{#4} \bigg| {#5} \right)}}
\newcommand{\assign}{:=}

\begin{document}

\title{A triple integral analog of a multiple zeta value}

\author{Tewodros Amdeberhan
\thanks{Department of Mathematics, Tulane University, New Orleans, LA 70118, \email{tamdeber@tulane.edu}}
\and Victor H. Moll
\thanks{Department of Mathematics, Tulane University, New Orleans, LA 70118, \email{vhm@tulane.edu}}
\and Armin Straub
\thanks{Department of Mathematics and Statistics, University of South Alabama, 411 University Blvd N, Mobile, AL 36688, \email{straub@southalabama.edu}}
\and Christophe Vignat
\thanks{Department of Mathematics, Tulane University, New Orleans, LA 70118, \email{cvignat@tulane.edu}}}

\date{April 13, 2020}

\maketitle

{\hfill\sl Dedicated to Bruce Berndt on the occasion of his 80\textsuperscript{th} birthday}

\bigskip

\begin{abstract}
We establish the triple integral evaluation
\[ \int_{1}^{\infty} \int_{0}^{1} \int_{0}^{1} \frac{dz \, dy \, dx}{x(x+y)(x+y+z)}
  = \frac{5}{24} \zeta(3), \]
as well as the equivalent polylogarithmic double sum
\[ \sum_{k=1}^{\infty} \sum_{j=k}^{\infty} \frac{(-1)^{k-1}}{k^{2}} \, \frac{1}{j \, 2^{j}}
  = \frac{13}{24} \zeta(3). \]
This double sum is related to, but less approachable than, similar sums
studied by Ramanujan. It is also reminiscent of Euler's formula $\zeta(2,1)
= \zeta(3)$, which is the simplest instance of duality of multiple
polylogarithms. We review this duality and apply it to derive a companion
identity. We also discuss approaches based on computer algebra. All of our
approaches ultimately require the introduction of polylogarithms and
nontrivial relations between them. It remains an open challenge to relate the
triple integral or the double sum to $\zeta(3)$ directly.
\end{abstract}

\section{Introduction}
\label{sec-intro1}

The evaluation 
\begin{equation}\label{eq:int:zeta3}
  \zeta(3) = \int_{0}^{1} \int_{x}^{1} \int_{y}^{1} \frac{dz \, dy \, dx}{(1-x)yz} 
\end{equation}
is given by M.~Kontsevich and D.~Zagier  \cite{kontsevich-2000a} as an
illustration that $\zeta(3)$ is a \emph{period}, in the sense that it is the
value of an absolutely convergent integral of a rational function with
rational coefficients, over a domain in $\mathbb{R}^{3}$ given by polynomial
inequalities $(0 < x < y < z < 1)$ with rational coefficients. The goal of
this work is to prove and discuss the following, much less obvious, variation of a triple integral evaluation.
\begin{theorem}\label{thm:Z3}
  We have
\begin{equation}\label{eq:Z3}
  Z_{3} := \int_{1}^{\infty} \int_{0}^{1} \int_{0}^{1} \frac{dz \, dy \, dx}{x(x+y)(x+y+z)}
  = \frac{5}{24} \zeta(3).
\end{equation}
\end{theorem}
In contrast to \eqref{eq:int:zeta3}, it appears to be a rather tricky problem to relate the triple integral \eqref{eq:Z3} to $\zeta(3)$ directly.  Indeed, all of our approaches to this integral have ultimately required the introduction of polylogarithms and nontrivial relations between them.  We give such a proof in Section~\ref{sec:proof}.

The integral \eqref{eq:Z3} might be seen as a continuous analog of multiple zeta values defined by 
\begin{equation}\label{eq:mzv}
  \zeta(s_{1},\ldots, s_{k})
  = \sum_{n_{1}>n_{2}> \cdots > n_{k} \geq 1} \frac{1}{n_{1}^{s_{1}}n_{2}^{s_{2}} \cdots n_{k}^{s_{k}}}.
\end{equation}
These sums were introduced by Euler. The reader is referred to the site \cite{hoffman-refs} maintained by M.~Hoffman for a large collection 
of papers related to these series.  For instance, when $k=3$, the multiple zeta values can be written as 
\begin{equation}
  \zeta(s_{1},s_{2},s_{3})
  = \sum_{p=1}^{\infty} \sum_{q=1}^{\infty} \sum_{r=1}^{\infty} \frac{1}{p^{s_{3}} (p+q)^{s_{2}} (p+q+r)^{s_{1}}},
\end{equation}
where the similarity with $Z_{3}$ becomes apparent. 

\begin{remark}\label{rk:hyperint}
Algorithmic approaches to computing (period) integrals such as the one in
\eqref{eq:Z3} are described in \cite{bb-feynman} and
\cite{panzer-hyperint}. In particular, Panzer implemented his symbolic
integration approach \cite{panzer-hyperint} using hyperlogarithms in a Maple
package called \texttt{HyperInt}. Using this package, the integral
\eqref{eq:Z3} is automatically evaluated as
\begin{equation*}
  Z_3 = \tfrac{19}{8} \zeta (3) - 2 \log (2) \zeta (2) - \Li_{2, 1}
   \left(\tfrac{1}{2}, 2 \right) - \Li_{1, 1, 1} \left(\tfrac{1}{3},
   \tfrac{3}{2}, 2 \right)
\end{equation*}
featuring the multiple polylogarithms reviewed in Section~\ref{sec:polylog}.
Simplifying the right-hand side to a multiple of $\zeta (3)$, however, is not
straightforward. Further comments on evaluating the integral $Z_3$ with the help of
computer algebra are included in Section~\ref{sec:computer}.
\end{remark}
 
In Section~\ref{sec:double-sum}, we show that the triple integral evaluation of Theorem~\ref{thm:Z3} is equivalent to the following relation between polylogarithms.
\begin{theorem}\label{thm:S3}
  We have
  \begin{equation}\label{eq:S3}
    S_3 := \sum_{k=1}^{\infty} \sum_{j=k}^{\infty} \frac{(-1)^{k-1}}{k^{2}} \, \frac{1}{j \, 2^{j}}
    = \frac{13}{24} \zeta(3).
  \end{equation}
\end{theorem}
As in the case of the triple integral, it is not clear how to relate this double sum directly to $\zeta(3)$.
In the notation for multiple polylogarithms reviewed in Section~\ref{sec:polylog}, the double sum in \eqref{eq:S3} can be expressed as
\begin{equation}\label{eq:S3:Li}
  S_3 = -\Li_{1,2}(\tfrac12, -1) - \Li_3(-\tfrac12).
\end{equation}
Readers familiar with multiple polylogarithms might therefore wonder whether the identity \eqref{eq:S3} is an instance of duality \cite[Section~6.1]{b3l}.
However, it appears that duality does not help in evaluating \eqref{eq:S3:Li}.  We review duality 
in Section~\ref{sec:duality} and show that it instead naturally provides the companion identity
\begin{equation}
  \sum_{\substack{
    k, m, n \geq 1\\
    \text{$k$ odd}
  }} \frac{1}{3^n}  \frac{1}{k (k + m) (k + m + n)} =
  \frac{13}{48} \zeta (3) . \label{eq:companion:intro}
\end{equation}

We note that \eqref{eq:S3} is reminiscent of Ramanujan's identity \cite[p.~259]{berndtI} (see also \cite[(7.4)]{b3l})
\begin{equation}\label{eq:S3:rama}
  \sum_{k=1}^\infty \sum_{j=k}^\infty \frac{1}{k} \, \frac{1}{j^2 \, 2^j}
  = \zeta(3) - \frac{\pi^2}{12} \log(2).
\end{equation}
Indeed, we show in Section~\ref{sec:double-sum:2} that Ramanujan's approach can be applied to evaluate
the non-alternating version of \eqref{eq:S3} as
\begin{equation}\label{eq:S3:plus}
  \sum_{k=1}^{\infty} \sum_{j=k}^{\infty} \frac{1}{k^{2}} \, \frac{1}{j \, 2^{j}}
  = \frac{5}{8} \zeta(3).
\end{equation}
We further indicate that \eqref{eq:S3} does not succumb readily to the same approach.

\begin{remark}\label{rk:Z2}
  We conclude this introduction by observing that the $2$-dimensional analog of \eqref{eq:Z3}, namely,
  \begin{equation}
    Z_2 := \int_{1}^{\infty} \int_{0}^{1} \frac{dy \, dx}{x(x+y)}
  \end{equation}
  is simple to evaluate.
  For instance, we can transform the domain of integration to the unit square via the change of variables $x \mapsto 1/x$ to obtain
  \begin{align*}
    Z_2
    = \int_0^1 \int_0^1 \frac{\id y \id x}{1 + x y}  
    = \sum_{n = 0}^{\infty} \int_0^1 \int_0^1 (- x y)^n \id y \id x  
    = \sum_{n = 0}^{\infty} \frac{(- 1)^n}{ (n + 1)^2} = \frac{1}{2} \zeta (2).
  \end{align*}
\end{remark}

\section{Polylogarithms}
\label{sec:polylog}
 
The polylogarithm function is defined, for $|z|<1$, by the power series
\begin{equation}\label{eq:Li}
  \Li_{s}(z) = \sum_{k=1}^{\infty} \frac{z^{k}}{k^{s}}.
\end{equation}
The function $\Li_2$ is also called the \emph{dilogarithm} and $\Li_3$ the \emph{trilogarithm}.
For other values of $z$, the polylogarithms are defined by analytic continuation, with the principal branch obtained from a cut along the positive real axis from $z=1$ to $\infty$.
We record that the analytic continuation of $\Li_{2}$ is provided by the integral representation
\begin{equation}\label{eq:Li2:int}
  \Li_{2}(z) = \int_{1}^{1-z} \frac{\log(t)}{1-t} dt.
\end{equation}

These functions satisfy a large collection of identities. Among these are the duplication formula
\begin{equation}\label{eq:Li:sign}
  \Li_{s}(-z) = -\Li_{s}(z) + 2^{1-s} \Li_{s}(z^{2}),
\end{equation}
as well as the inversion formula
\begin{equation}
  \Li_n(z) + (-1)^n \Li_n(1/z)
  = -\frac{(2\pi i)^n}{n!} B_n\left( \frac12 + \frac{\log(-z)}{2\pi i} \right),
\end{equation}
which holds for positive integers $n$ and $z\not\in [0,1]$.  Here, the $B_n$ are the Bernoulli polynomials, of which we will only use $B_2(x) = x^2-x+\tfrac16$ and $B_3(x) = x^3 - \tfrac32 x^2 + \tfrac12 x$.  In the cases $n=2$ and $n=3$ the inversion formula thus becomes
\begin{align}
  \Li_{2}(z) &= -\Li_{2} (1 /z ) - \tfrac{1}{2} \log^{2}(-z) - \tfrac{\pi^2}{6}, \label{eq:Li2:inv} \\
  \Li_{3}(z) &= \Li_{3} (1 /z ) - \tfrac{1}{6} \log^{3}(-z) - \tfrac{\pi^2}{6} \log(-z). \label{eq:Li3:inv} 
\end{align}
The dilogarithm and the trilogarithm satisfy a number of additional relations.
A small selection of formul{\ae} is recorded in \cite[Section 25.12]{olver-2010a} (see also \cite[Chapter~9]{berndtI}). A more complete collection appears in  \cite{lewin-1981a} and \cite{lewin-1981b}.
For our purposes, we record the reflection formul{\ae}
\begin{equation}\label{eq:Li2:refl}
  \Li_{2}(z) + \Li_{2}(1-z) = \tfrac{\pi^2}{6} - \log(z) \, \log(1-z)
\end{equation}
and
\begin{align}\label{eq:Li3:refl}
\begin{split}
  \Li_{3}(z) + &\Li_{3} (1 -z ) +  \Li_{3} (1-1 /z ) \\
  & = \zeta(3) + \tfrac16 \log^{3}(z) + \tfrac{\pi^2}{6} \log(z) - \tfrac12 \log^{2}(z) \, \log(1-z),
\end{split}
\end{align}
as well as the identity \cite[(6.34)]{lewin-1981a}
\begin{align}\label{eq:Li3:id}
  \Li_3 \left(\tfrac{1 - z}{1 + z} \right) &- \Li_3 \left(- \tfrac{1 - z}{1 + z} \right)
  = \tfrac{1}{2} \Li_3 \left(\tfrac{- z^2}{1 - z^2} \right) - 2 \Li_3 \left(\tfrac{- z}{1 - z} \right) - 2 \Li_3 \left(\tfrac{z}{1 + z} \right)
  \\\nonumber
  & + \tfrac{7}{4} \zeta (3) + \tfrac{\pi^2}{4} \log \left(\tfrac{1 - z}{1 + z} \right) +
  \tfrac{1}{4} \log^2 \left(\tfrac{1 + z}{1 - z} \right) \log \left(\tfrac{1
  - z^2}{z^2} \right).
\end{align}
It follows from \eqref{eq:Li2:refl} and \eqref{eq:Li3:refl}, together with $\Li_3 (- 1) = - \tfrac{3}{4} \zeta (3)$, that
\begin{align}
  \Li_{2} ( \tfrac{1}{2} ) &= \tfrac{\pi^2}{12} - \tfrac12 \log^2(2), \label{eq:Li2:2}\\
  \Li_{3} ( \tfrac{1}{2} ) &= \tfrac{7}{8} \zeta(3) - \tfrac{\pi^2}{12} \log(2) + \tfrac{1}{6}\log^{3}(2). \label{eq:Li3:2}
\end{align}
Combining the mentioned identities as done, for instance, in \cite[(39.1), (39.4); p.~324]{berndtIV} where Ramanujan considers such combinations, we can further derive the relation
\begin{equation}\label{eq:Li2:3}
  2\Li_2(\tfrac13) - \Li_2(-\tfrac13) = \tfrac{\pi^2}{6} - \tfrac12 \log^2(3).
\end{equation}
Similarly, by setting $z = 1 / 2$ in \eqref{eq:Li3:id}, we find
\begin{equation}\label{eq:Li3:3}
  2 \Li_3 (\tfrac{1}{3} ) - \Li_3 (- \tfrac{1}{3} )
  = \tfrac{13}{6} \zeta (3) - \tfrac{\pi^2}{6} \log (3) + \tfrac{1}{6} \log^3 (3) .
\end{equation}

Generalizing \eqref{eq:Li}, the \emph{multiple polylogarithms} are the sums
\begin{equation}\label{eq:Li:mult}
  \Li_{s_{1},\ldots, s_{k}}(z_1,\ldots,z_k)
  = \sum_{n_{1}>n_{2}> \cdots > n_{k} \geq 1} \frac{z_1^{n_1} z_2^{n_2} \cdots z_k^{n_k}}{n_{1}^{s_{1}}n_{2}^{s_{2}} \cdots n_{k}^{s_{k}}},
\end{equation}
though notation varies throughout the literature (in particular, the order of the arguments is sometimes reversed; the choice here is consistent with multiple zeta values as defined in \eqref{eq:mzv}, so that \eqref{eq:Li:mult} becomes $\zeta(s_1,\ldots,s_k)$ when $z_1=z_2=\ldots=z_k=1$).
A vast literature, see, for instance, \cite{b3l} and the references therein, exists concerning relations between values of multiple polylogarithms.

\section{A proof of the triple integral evaluation}
\label{sec:proof}

In this section, we evaluate the triple integral $Z_3$ as claimed in Theorem~\ref{thm:Z3}.
We begin by integrating with respect to $x$ and, then, with respect to $y$ to obtain
\begin{align}
  Z_{3} &= \int_{0}^{1} \int_{0}^{1} \left[ \frac{\log(y+1)}{yz} - \frac{\log(y+z+1)}{z(y+z)} \right] \, dy \, dz \nonumber\\
  &= \int_{0}^{1} \left[ \Li_{2}(-z-1) - \Li_2(-1) - \Li_{2}(-z) \right] \frac{dz}{z} \label{eq:Z3:int:Li}\\
  &= \int_{0}^{1} \left[ \Li_{2}(-z-1) - \Li_2(-1) \right] \frac{dz}{z} + \frac34 \zeta(3). \nonumber
\end{align}
For the second equality, we used the fact that the derivative of $\Li_2(-x)$ is  $- \log(x+1)/x$,
while, for the third, we integrated term by term to find
\begin{equation}\label{eq:zeta3:alt}
  \int_0^1 \Li_2(-z) \frac{dz}{z}
  = \int_0^1 \sum_{n=1}^\infty \frac{(-z)^n}{n^2} \frac{dz}{z}
  = \sum_{n=1}^\infty \frac{(-1)^n}{n^3}
  = -\frac34 \zeta(3).
\end{equation}
We now use the integral representation \eqref{eq:Li2:int} of the dilogarithm to conclude
\begin{align*}
  Z_3 &= \int_{0}^{1} \int_{2}^{2+z} \frac{\log(t) \, dt}{1-t} \frac{dz}{z} + \frac34 \zeta(3).
\end{align*}
Exchanging the order of integration and evaluating the inner integral, we obtain 
\begin{equation}\label{eq:Z3:int:log2}
  Z_3 = \int_0^1 \frac{\log(t) \, \log(t+2)}{1+t} \, dt + \frac34 \zeta(3).
\end{equation}
We now replace
\begin{equation}
  2\log(t) \log(t+2) = \log^{2}(t) + \log^{2} (t+2) - \log^{2} \left( \frac{t}{t+2} \right)
\end{equation}
in the integral \eqref{eq:Z3:int:log2},
and make the change of variables $t/(t+2) \to t$ in the third of the three resulting integrals, to find
\begin{align}
  Z_3 &= \frac12 \int_0^1 \frac{\log^2(t)}{1+t} \, dt
  + \frac12 \int_0^1 \frac{\log^2(t+2)}{1+t} \, dt
  - \int_0^{1/3} \frac{\log^2(t)}{1-t^2} \, dt
  + \frac34 \zeta(3) \nonumber\\
  &= \frac12 \int_0^1 \frac{\log^2(t+2)}{1+t} \, dt
  - \int_0^{1/3} \frac{\log^2(t)}{1-t^2} \, dt
  + \frac32 \zeta(3). \label{eq:Z3:int:2}
\end{align}
Here, we evaluated the first integral by expanding $1/(1+t)$ as a geometric series, integrating term by term and using
\begin{equation}\label{eq:log2:mellin}
  \int_0^x t^{n-1} \log^2(t) \, dt
  = \frac{x^{n}}{n^3} \left( 2 - 2n \log(x) + n^2 \log^2(x) \right)
\end{equation}
(which is readily verified by differentiating both sides)
as well as the final sum in \eqref{eq:zeta3:alt}.
Summing \eqref{eq:log2:mellin}, we find
\begin{align}
  \int_0^x \frac{\log^2(t)}{1-t^2} \, dt
  &= \Li_3(x)-\Li_3(-x) - \log(x) (\Li_2(x)-\Li_2(-x))
  \\\nonumber&\quad
  + \tfrac12 \log^2(x) \log( \tfrac{1+x}{1-x} ).
\end{align}
In light of the relations \eqref{eq:Li2:3} and \eqref{eq:Li3:3}, this shows that
\begin{equation}\label{eq:int:log2:3}
  \int_0^{1/3} \frac{\log^2(t)}{1-t^2} \, dt
  = \tfrac{13}{6} \zeta(3) - \Li_3(\tfrac13) - \Li_2(\tfrac13) \log(3)
  - \tfrac16 \log(\tfrac98)\log^2(3).
\end{equation}
By appealing to analytic continuation, we could similarly approach the remaining integral in \eqref{eq:Z3:int:2}.
Alternatively, we use the formula \cite[(6.28)]{lewin-1981b}
\begin{align}\label{lewin-1}
  \int_{0}^{t} \frac{\log^{2}(u+1)}{u} du
  &= \log t \log ^{2}(t+1) - \tfrac{2}{3} \log ^{3}(t+1) \\\nonumber
  &- 2 \log(t+1) \Li_{2} \left( \frac{1}{t+1} \right) - 2 \Li_{3} \left( \frac{1}{t+1} \right) + 2 \zeta(3)
\end{align}
to evaluate
\begin{align}
  \frac12 \int_0^1 \frac{\log^2(t+2)}{1+t} \, dt
  &= \frac12\int_{0}^{2} \frac{\log^{2}(t+1)}{t} \, dt - \frac12\int_{0}^{1} \frac{\log^{2}(t+1)}{t} \, dt \nonumber\\
  &= \tfrac{7}{8} \zeta(3) - \Li_3(\tfrac13) - \Li_2(\tfrac13) \log(3)
  - \tfrac16 \log(\tfrac98)\log^2(3). \label{eq:int:log2:s}
\end{align}
Here, we used \eqref{eq:Li2:2} and \eqref{eq:Li3:2} to reduce the polylogarithms at $1/2$.
The claimed evaluation, $Z_3 = \frac78 \zeta(3) - \frac{13}{6} \zeta(3) + \frac32 \zeta(3) = \frac{5}{24} \zeta(3)$, now follows from using \eqref{eq:int:log2:3} and \eqref{eq:int:log2:s} in \eqref{eq:Z3:int:2}.

\section{An equivalent double sum evaluation}
\label{sec:double-sum}

In this section, we relate the triple integral $Z_3$ of Theorem~\ref{thm:Z3} to the double sum in Theorem~\ref{thm:S3} in the following way.

\begin{lemma}\label{lem:Z3:S3}
  We have
  \begin{equation}\label{eq:Z3:S3}
    Z_{3} =  \frac{3}{4} \zeta(3) - \sum_{k=1}^{\infty} \sum_{j=k}^{\infty} \frac{(-1)^{k-1}}{k^{2}} \frac{1}{j \, 2^{j}}.
  \end{equation}
\end{lemma}

\begin{proof}
  Expanding the integrand of \eqref{eq:Z3} as a geometric series, we obtain
  \begin{align*}
    \int_{0}^{1} \frac{dz}{x(x+y)(x+y+z) }
    & = \frac{1}{x(x+y)^{2}} \int_{0}^{1} \frac{dz}{1 + \frac{z}{x+y}} \\
    & = \frac{1}{x} \sum_{k=0}^{\infty} \int_{0}^{1} \frac{(-1)^{k} \, z^{k} \, dz}{(x+y)^{k+2}} \\
    & = \frac{1}{x} \sum_{k=0}^{\infty} \frac{(-1)^{k}}{(k+1) (x+y)^{k+2}}.
  \end{align*}
  We then integrate with respect to $y$ to find
  \begin{align*} 
    \int_{0}^{1} \int_{0}^{1} \frac{dz \, dy}{x(x+y)(x+y+z)}
    & = \frac{1}{x} \int_{0}^{1} \sum_{k=0}^{\infty} \frac{(-1)^{k} \, dy}{(k+1) (x+y)^{k+2}} \\
    & = \frac{1}{x} \sum_{k=0}^{\infty} \frac{(-1)^{k}}{(k+1)^{2}} \left[ \frac{1}{x^{k+1}} - \frac{1}{(x+1)^{k+1}} \right] \\
    & = \sum_{k=0}^{\infty} \frac{(-1)^{k}}{(k+1)^{2} x^{k+2}} - \frac{1}{x} \sum_{k=0}^{\infty} \frac{(-1)^{k}}{(k+1)^{2}} \frac{1}{(x+1)^{k+1}}.
  \end{align*}
  Integrating the first term with respect to $x$ results in
  \begin{equation}
    \sum_{k=0}^{\infty} \frac{(-1)^{k}}{(k+1)^{2} } \int_{1}^{\infty} \frac{dx}{x^{k+2}} = \sum_{k=0}^{\infty} \frac{(-1)^{k}}{(k+1)^{3}} = 
    \frac{3}{4} \zeta(3).
  \end{equation}
  On the other hand, the second term contributes
  \begin{equation}
    \sum_{k=0}^{\infty} \frac{(-1)^{k}}{(k+1)^{2}} \int_{1}^{\infty} \frac{dx}{x \, (x+1)^{k+1}}.
  \end{equation}
  Finally, we note that 
  \begin{equation}
    \int_{1}^{\infty} \frac{dx}{x(x+1)^{k+1}}  =  \sum_{j=k}^{\infty} \frac{1}{j \, 2^{j}},
  \end{equation}
  which follows from the change of variables $w = 1/(x+1)$ and expanding a geometric series.
  Combining terms, we arrive at \eqref{eq:Z3:S3}.
\end{proof}

Observe that Lemma~\ref{lem:Z3:S3} and Theorem~\ref{thm:Z3} immediately imply Theorem~\ref{thm:S3}.

Let us conclude this section by deriving an alternative integral representation for the double sum
\begin{equation}
  S_3 = \sum_{k = 1}^{\infty} \frac{(- 1)^{k - 1}}{k^2} \sum_{j=k}^{\infty}  \frac{1}{j \, 2^{j}}
\end{equation}
of Theorem~\ref{thm:S3}.

\begin{theorem}
  We have
\begin{equation}\label{int-dilog}
  S_3
  = - \int_{0}^{1/2} \frac{\Li_{2}(-t)}{t(1-t)} \, dt
  = - 4 \int_{0}^{\pi/4} \frac{ \Li_{2}( - \sin^{2} \varphi)}{\sin(2 \varphi)} \, d \varphi.
\end{equation}
\end{theorem}

\begin{proof}
  It is not hard to show (and can even be done automatically, for instance, using creative telescoping) that, for positive integers $k$, the inner sum can be expressed in hypergeometric terms as
  \begin{equation}
    \sum_{j=k}^{\infty}  \frac{1}{j \, 2^{j}}
    = \log 2 - \sum_{j=1}^{k-1} \frac{1}{j \, 2^{j}}
    = \frac{1}{k} \pFq21{k \quad k}{k+1}{-1}.
  \end{equation}
  The double sum therefore equals
  \begin{equation}
    S_3 = \sum_{k=1}^{\infty} \frac{(-1)^{k-1}}{k^{3}}  \pFq21{k \quad k}{k+1}{-1}.
  \end{equation}
  The classical integral representation for the hypergeometric function 
  \begin{equation}
    \pFq21{a \quad b}{c}{z} = \frac{1}{B(b,c-b)} \int_{0}^{1} x^{b-1} (1-x)^{c-b-1} (1 - xz)^{-a} \, dx,
  \end{equation}
  see, for instance, \cite[page 65]{andrews-1999a}, gives 
  \begin{equation}
    S_3 =  \int_{0}^{1}\left[  \sum_{k=1}^{\infty} \frac{(-1)^{k-1}}{k^{2}}  \left( \frac{x}{x+1} \right)^{k}  \right] \, \frac{dx}{x}.
  \end{equation}
  The series definition \eqref{eq:Li} for the dilogarithm, followed by the change of variables $t = \frac{x}{x+1}$, therefore implies the claimed integral representation.
\end{proof}

We note that, alternatively, the series $S_3$ can be obtained starting from the second integral in \eqref{int-dilog} by expanding the dilogarithm and using \cite[Entry 2.516.2]{gradshteyn-2015a}
\begin{equation}
\int \frac{\sin^{2k-1} \varphi}{\cos \varphi} \, d \varphi = - \sum_{j=1}^{k-1} \frac{\sin^{2j} \varphi }{2j} - \log( \cos \varphi).
\end{equation}

\section{A companion identity via duality}
\label{sec:duality}

Recall that Theorem~\ref{thm:S3} is equivalent to
\begin{equation}
  \label{eq:S3:Li:only} - \operatorname{Li}_{2, 1} \left(- 1, \tfrac{1}{2} \right) -
  \operatorname{Li}_3 \left(- \tfrac{1}{2} \right) = \tfrac{13}{24} \zeta (3) .
\end{equation}
In this section, we review duality \cite[Section~6.1]{b3l} and apply it to
the identity \eqref{eq:S3:Li:only} to deduce the following companion identity.

\begin{theorem}\label{thm:companion}
  We have
  \begin{equation}
    \sum_{\substack{
      k, m, n \geq 1\\
      \text{$k$ odd}
    }} \frac{1}{3^n}  \frac{1}{k (k + m) (k + m + n)} =
    \frac{13}{48} \zeta (3) . \label{eq:companion}
  \end{equation}
\end{theorem}

We use the notation of \cite{b3l} and write
\begin{equation*}
  l \left(\begin{array}{c}
     s_1, \ldots, s_k\\
     y_1, \ldots, y_k
   \end{array} \right) \assign \operatorname{Li}_{s_1, \ldots, s_k} \left(\frac{1}{y_1}, \frac{y_1}{y_2}, \frac{y_2}{y_3}, \ldots, \frac{y_{k -
   1}}{y_k} \right)
\end{equation*}
for multiple polylogarithms as well as
\begin{equation*}
  \int_0^1 \omega (a_1) \omega (a_2) \cdots \omega (a_n) \assign \int_0^1
   \int_0^{x_1} \cdots \int_0^{x_{n - 1}} \frac{\id x_n \cdots \id x_2
   \id x_1}{(x_n - a_n) \cdots (x_2 - a_2) (x_1 - a_1)}
\end{equation*}
for iterated integrals. These integrals provide a natural way to express
multiple polylogarithms. Indeed, we have the weight-dimensional iterated
integral representation \cite[(4.9)]{b3l}
\begin{equation}
  \label{eq:li:itint:1} l \left(\begin{array}{c}
    s_1, \ldots, s_k\\
    y_1, \ldots, y_k
  \end{array} \right) = (- 1)^k \int_0^1 \prod_{j = 1}^k \omega (0)^{s_j - 1}
  \omega (y_j).
\end{equation}
Note that the weight of the multiple polylogarithm is defined to be $s_1 + s_2
+ \cdots + s_k$ (and that this matches the number of integrations in the
iterated integral representation \eqref{eq:li:itint:1}).

Reversing the order of integration in \eqref{eq:li:itint:1} and replacing each
integration variable $x$ by $1 - x$ results in the dual iterated integral
representation \eqref{eq:li:itint:2} of the multiple polylogarithm.

\begin{theorem}[{\cite[(6.1)]{b3l}}]
  We have
  \begin{equation}
    \label{eq:li:itint:2} l \left(\begin{array}{c}
      s_1, \ldots, s_k\\
      y_1, \ldots, y_k
    \end{array} \right) = (- 1)^{s_1 + s_2 + \cdots + s_k + k} \int_0^1
    \prod_{j = 0}^{k - 1} \omega (1 - y_{k - j}) \omega (1)^{s_{k - j} - 1}.
  \end{equation}
\end{theorem}

A famous instance of duality is Euler's formula (see, for instance, \cite{borweinj-2006a})
\begin{align}
  \zeta (2, 1) = l \left(\begin{array}{cc}
     2 & 1\\
     1 & 1
   \end{array} \right) &= \int_0^1 \omega (0) \omega (1) \omega (1)
   \label{eq:zeta21}\\\notag &=
   - \int_0^1 \omega (0) \omega (0) \omega (1) = l \left(\begin{array}{c}
     3\\
     1
   \end{array} \right) = \zeta (3),
\end{align}
where both \eqref{eq:li:itint:1} and \eqref{eq:li:itint:2} have been used to
relate one polylogarithm to its dual.

\begin{proof}[Proof of Theorem~\ref{thm:companion}]
  Similar to \eqref{eq:zeta21}, for the polylogarithms involved in \eqref{eq:S3:Li:only}, we find
  that
  \begin{align*}
    \operatorname{Li}_{1, 2} \left(\tfrac{1}{2}, - 1 \right)
    & = l \left(\begin{array}{cc}
      1 & 2\\
      2 & - 2
    \end{array} \right)
    = \int_0^1 \omega (2) \omega (0) \omega (- 2)\\
    & = - \int_0^1 \omega (3) \omega (1) \omega (- 1)
    = l \left(\begin{array}{ccc}
      1 & 1 & 1\\
      3 & 1 & - 1
    \end{array} \right)
    = \operatorname{Li}_{1, 1, 1} \left(\tfrac{1}{3}, 3, - 1 \right)\\
    & = \sum_{n_1 > n_2 > n_3 \geq 1} \frac{(- 1)^{n_3}}{3^{n_1 -
    n_2}}  \frac{1}{n_1 n_2 n_3},
  \end{align*}
  as well as
  \begin{align*}
    \operatorname{Li}_3 \left(- \tfrac{1}{2} \right)
    & = l \left(\begin{array}{c}
      3\\
      - 2
    \end{array} \right)
    = - \int_0^1 \omega (0)^2 \omega (- 2)\\
    & = \int_0^1 \omega (3) \omega (1)^2 = - l \left(\begin{array}{ccc}
      1 & 1 & 1\\
      3 & 1 & 1
    \end{array} \right)
    = - \operatorname{Li}_{1, 1, 1} \left(\tfrac{1}{3}, 3, 1 \right)\\
    & = - \sum_{n_1 > n_2 > n_3 \geq 1} \frac{1}{3^{n_1 - n_2}} 
    \frac{1}{n_1 n_2 n_3} .
  \end{align*}
  Taken together, we obtain
  \begin{equation*}
    \sum_{n_1 > n_2 > n_3 \geq 1} \frac{1}{3^{n_1 - n_2}}  \frac{1}{n_1
     n_2 n_3} - \sum_{n_1 > n_2 > n_3 \geq 1} \frac{(- 1)^{n_3}}{3^{n_1 -
     n_2}}  \frac{1}{n_1 n_2 n_3} = \frac{13}{24} \zeta (3)
  \end{equation*}
  or, equivalently, \eqref{eq:companion}.
\end{proof}

\section{A related simpler double sum}
\label{sec:double-sum:2}

Ramanujan derived \eqref{eq:S3:rama} by evaluating \cite[Entry~9;
p.~251]{berndtI} the function
\begin{equation*}
  g (z) = \sum_{k = 1}^{\infty} H_k  \frac{z^{k + 1}}{(k + 1)^2},
  \quad\text{where }\;
  H_k = \sum_{j = 1}^k \frac{1}{j},
\end{equation*}
in polylogarithmic terms as
\begin{equation}
  \label{eq:g:Li} g (1 - z) = \tfrac{1}{2} \log^2 (z) \log (1 - z) +
  \Li_2 (z) \log (z) - \Li_3 (z) + \zeta (3),
\end{equation}
and observing that the double sum \eqref{eq:S3:rama} is given by $g (1 / 2) +
\Li_3 (1 / 2)$. We next show that \eqref{eq:S3:plus} can be derived in
an analogous manner.

\begin{lemma}
  For $| z | < 1$, define
  \begin{equation*}
    h (z) = \sum_{k = 1}^{\infty} H_k^{(2)}  \frac{z^{k + 1}}{k + 1},
    \quad\text{where }\;
     H_k^{(2)} = \sum_{j = 1}^k \frac{1}{j^2} .
  \end{equation*}
  Then,
  \begin{equation}
    \label{eq:h:Li} h (1 - z) = 2 \Li_3 (z) - \Li_2 (z) \log (z) -
    \tfrac{\pi^2}{6} \log (z) - 2 \zeta (3) .
  \end{equation}
\end{lemma}

\begin{proof}
  We note that
  \begin{equation*}
    h' (z) = \sum_{k = 1}^{\infty} H_k^{(2)} z^k = \frac{\Li_2 (z)}{1 -
     z} .
  \end{equation*}
  Integrating by parts, we find
  \begin{equation*}
    h (z) = - \log (1 - z) \Li_2 (z) - \int_0^z \frac{\log^2 (1 -
     t)}{t} \id t = - \log (1 - z) \Li_2 (z) - 2 g (z),
  \end{equation*}
  so that the claim follows from \eqref{eq:g:Li} combined with
  \eqref{eq:Li2:refl}.
\end{proof}

It is now straightforward to deduce \eqref{eq:S3:plus}.

\begin{corollary}\label{cor:S3:plus}
  We have
  \begin{equation*}
    \sum_{k = 1}^{\infty} \sum_{j = k}^{\infty} \frac{1}{k^2} 
     \, \frac{1}{j \, 2^j} = \frac{5}{8} \zeta (3) .
  \end{equation*}
\end{corollary}

\begin{proof}
  Observe that
  \begin{equation*}
    \sum_{k = 1}^{\infty} \sum_{j = k}^{\infty} \frac{1}{k^2} 
     \, \frac{1}{j \, 2^j} = h \left(\tfrac{1}{2}
     \right) + \Li_3 \left(\tfrac{1}{2} \right),
  \end{equation*}
  so that the claimed evaluation follows from \eqref{eq:h:Li} together with
  \eqref{eq:Li2:2} and \eqref{eq:Li3:2}.
\end{proof}

Let us indicate that similarly approaching the double sum $S_3$ in
Theorem~\ref{thm:S3} is not sufficient to evaluate it. Define, for $| z | <
1$,
\begin{equation*}
  \tilde{h} (z) = \sum_{k = 1}^{\infty} \tilde{H}_k^{(2)}  \frac{z^{k + 1}}{k + 1},
    \quad\text{where }\;
   \tilde{H}_k^{(2)} = \sum_{j = 1}^k \frac{(- 1)^{j - 1}}{j^2},
\end{equation*}
so that
\begin{equation*}
  S_3 = \tilde{h} \left(\tfrac{1}{2} \right) - \Li_3 \left(-
   \tfrac{1}{2} \right) .
\end{equation*}
Observing that
\begin{equation*}
  \tilde{h}' (z) = \sum_{k = 1}^{\infty} \tilde{H}_k^{(2)} z^k = -
   \frac{\Li_2 (- z)}{1 - z},
\end{equation*}
we then find
\begin{equation}
  \label{eq:h1:Li} \tilde{h} (z) = \log (1 - z) \Li_2 (- z) + \int_0^z
  \frac{\log (1 - t) \log (1 + t)}{t} \id t.
\end{equation}
In particular,
\begin{equation}
  \label{eq:S3:h1:Li} S_3 = \int_0^{1 / 2} \frac{\log (1 - t) \log (1 + t)}{t}
  \id t - \Li_3 \left(- \tfrac{1}{2} \right) - \log (2) \Li_2
  \left(- \tfrac{1}{2} \right) .
\end{equation}
The integral in \eqref{eq:h1:Li} can be expressed as a sum of polylogarithms
at various arguments, though the result takes a rather more complicated form than
\eqref{eq:h:Li}. As a consequence, it still requires a considerable amount of
polylogarithmic relations to deduce the simple evaluation $S_3 =
\tfrac{13}{24} \zeta (3)$ from \eqref{eq:S3:h1:Li}.

\section{Computer algebraic approaches}
\label{sec:computer}

We indicated in Remark~\ref{rk:hyperint} that (period) integrals such as
\begin{equation}\label{eq:Z3:cas}
  Z_{3} = \int_{1}^{\infty} \int_{0}^{1} \int_{0}^{1} \frac{dz \, dy \, dx}{x(x+y)(x+y+z)}
\end{equation}
can be algorithmically evaluated in terms of
multiple polylogarithms. These evaluations, however, are not typically
in simplified form and establishing the necessary relations between special values of
polylogarithms can be (sometimes prohibitively) difficult. For instance,
Theorem~\ref{thm:S3} is equivalent to the polylogarithmic relation
\begin{equation}
  \label{eq:S3:Li:0} - \Li_{1, 2} \left(\tfrac{1}{2}, - 1 \right) -
  \Li_3 \left(- \tfrac{1}{2} \right) = \tfrac{13}{24} \zeta (3) .
\end{equation}
We note that $- \Li_{1, 2} (x, - 1) = H (1, - 2 ; x)$ is an instance of
the harmonic polylogarithm introduced in \cite{rm-hpl}. A Mathematica
implementation of algorithms for working with and simplifying harmonic
polylogarithms is provided by Maitre \cite{hpl}. For instance, it is
possible to algorithmically convert the harmonic polylogarithm $H (1, - 2 ;
x)$ to trilogarithms $\Li_3 (x)$ and lower order terms:
\begin{align*}
  H (1, - 2 ; x)
  & = \Li_3 (- x) - \Li_3 (1 - x) + \Li_3 \left(\tfrac{x}{1 + x} \right) + \Li_3 \left(\tfrac{1 - x}{2} \right)\\
  & + \Li_3 \left(\tfrac{1 + x}{2} \right) - \Li_3 \left(\tfrac{2 x}{x - 1} \right) - \Li_3 \left(\tfrac{2 x}{1 + x} \right) - \Li_2 (x) \log (x + 1)\\
  & + \tfrac{1}{6} \log \left(\tfrac{1 - x}{8} \right) \log^2 (1 - x) - \tfrac{1}{2} \log (x) \log^2 (1 - x)\\
  & + \tfrac{\log^2 (2)}{2} \log (1 - x^2) - \tfrac{\log (2)}{2} \log^2 (1 + x) + \tfrac{\pi^2}{12} \log \left(\tfrac{1 - x}{1 + x} \right)\\
  & - \tfrac{3}{4} \zeta (3) - \tfrac{\log^3 (2)}{3} + \tfrac{\pi^2}{6}
  \log (2) .
\end{align*}
With a bit of human effort, we can then prove \eqref{eq:S3:Li:0} by employing
relations for the tri- and dilogarithm including those mentioned in
Section~\ref{sec:polylog}.

The question of whether an integral like \eqref{eq:Z3:cas}
can be automatically evaluated by a general purpose computer algebra system evolves in time.
For instance, given \eqref{eq:Z3:cas}, \textit{Maple} 18 evaluates two of three integrals and returns a single integral similar to \eqref{eq:Z3:int:Li}.
On the other hand, \textit{Mathematica} 9 returns the integral \eqref{eq:Z3:cas} unevaluated, while \textit{Mathematica} 10 produces
\begin{align*}
  Z_3 & = \Li_3 \left(- \tfrac{1}{3} \right) - 2 \Li_3 \left(\tfrac{1}{3} \right) + \tfrac{19}{8} \zeta (3) + \tfrac{1}{2} \log (3) \Li_2 \left(\tfrac{1}{9} \right) - 3 \log (3) \Li_2 \left(\tfrac{1}{3} \right)\\
  &\quad - \tfrac{1}{3} \log^3 (3) + \pi i \left[ \tfrac{1}{2} \Li_2 \left(\tfrac{1}{9} \right) - 3 \Li_2 \left(\tfrac{1}{3} \right) - \tfrac{1}{2} \log^2 (3) + \tfrac{\pi^2}{6} \right] .
\end{align*}
Assuming that this calculation is correct, and observing that $Z_{3}$ is real, we can conclude the relation
\begin{equation}\label{eq:Li2:9}
  \Li_{2} \left( \tfrac{1}{9} \right) = 6 \Li_{2} \left( \tfrac{1}{3} \right) + \log^{2}(3) - \tfrac{\pi^2}{3},
\end{equation}
which follows from \eqref{eq:Li:sign} combined with \eqref{eq:Li2:3} (see also \cite[(39.4); p.~324]{berndtIV}).
Using \eqref{eq:Li2:9}, the expression for the triple integral reduces to 
\begin{equation}\label{value-z3}
  Z_3 = \Li_3 \left(- \tfrac{1}{3} \right) - 2 \Li_3 \left(\tfrac{1}{3} \right) + \tfrac{19}{8} \zeta (3) - \tfrac{\pi^2}{6} \log (3) + \tfrac{1}{6} \log^3 (3).
\end{equation}
The remaining polylogarithms can be simplified using \eqref{eq:Li3:3}, which results in the desired evaluation
\begin{equation}
  Z_{3} = \tfrac{5}{24} \zeta(3),
\end{equation}
which we established in Theorem~\ref{thm:Z3}.

It is an interesting phenomenon that symmetrizing a problem can occasionally make it considerably more tractable by computer algebra.  An impressive instance is \cite{paule-rr}, where Paule significantly reduces the order of recurrences for certain $q$-sums by ``creative symmetrizing''.
In the remainder of this section, we indicate that a symmetrizing transformation also makes the integral $Z_3$ more palatable for symbolic evaluation.
To begin with, the change of variables $x \mapsto 1/x$ transforms the integral \eqref{eq:Z3:cas} to the unit cube:
\begin{equation}\label{triple-1a}
  Z_{3} = \int_{0}^{1} \int_{0}^{1} \int_{0}^{1} \frac{x \, dx \, dy \, dz}{(1+xy)(1+ x(y+z))}.
\end{equation}
Denote the integrand of \eqref{triple-1a} by $u(x,y,z)$.
The symmetrization of the integral is defined in terms of 
\begin{equation}
  u^{\operatorname{sym}}(x,y,z) = \frac{1}{3} \left[ u(x,y,z) + u(y,z,x) + u(z,x,y) \right]
\end{equation}
as 
\begin{equation}\label{eq:Z3:sym}
  Z_{3} = \int_0^1 \int_{0}^{1} \int_{0}^{1} u^{\operatorname{sym}}(x,y,z) \, dx \, dy \, dz.
\end{equation}           
This form seems to be more favorable to a symbolic calculation. Indeed, \textit{Mathematica} is able to directly evaluate the symmetrized integral \eqref{eq:Z3:sym} as 
\begin{equation}
  Z_{3} = \tfrac{5}{24} \zeta(3) - \tfrac{\pi i}{6} 
  \left[ \Li_{2} \left( \tfrac{1}{9} \right) - 6 \Li_{2} \left( \tfrac{1}{3} \right) - \log^{2}(3) + \tfrac{\pi^2}{3} \right].
\end{equation}
Identity \eqref{eq:Li2:9}, or simply observing that \eqref{eq:Z3:sym} is real, then gives the value $Z_{3} = \tfrac{5}{24}\zeta(3)$ of Theorem~\ref{thm:Z3}.

\section{Conclusions}
\label{sec:conclusions}

We discussed several approaches to the triple integral \eqref{eq:Z3} as well as the equivalent double sum \eqref{eq:S3}.
However, in each case, we eventually required the introduction of polylogarithms and nontrivial relations between them.
It remains an interesting challenge to relate either of \eqref{eq:Z3} and \eqref{eq:S3} to $\zeta(3)$ directly.

In another direction, let us note that the unit square integral for $\frac12 \zeta(2)$ in Remark~\ref{rk:Z2}, a $2$-dimensional analog of the integral \eqref{eq:Z3},
has other simpler higher-dimensional generalizations.
Indeed, for any $m \ge 2$, we have the following integrals evaluating in terms of $\zeta(m)$:
\begin{equation*}
  \int_0^1 \cdots \int_0^1 \frac{\id x_1 \cdots \id x_m}{1 - x_1 \cdots
   x_m} = \zeta (m), \quad \int_0^1 \cdots \int_0^1 \frac{\id x_1 \cdots
   \id x_m}{1 + x_1 \cdots x_m} = (1 - 2^{1 - m}) \zeta (m) .
\end{equation*}
In particular,
\begin{equation*}
  \int_0^1 \int_0^1 \int_0^1 \frac{\id x \id y \id z}{1 - x y z} =
   \zeta (3), \quad \int_0^1 \int_0^1 \int_0^1 \frac{\id x \id y \id
   z}{1 + x y z} = \frac{3}{4} \zeta (3) .
\end{equation*}
Other interesting triple integrals involving $\zeta(3)$ have been considered in the literature.
For instance, in his proof, inspired by Ap\'ery, of the irrationality
of $\zeta (3)$, Beukers \cite{beukers1} considers the integrals
\begin{equation*}
  J_n = \frac{1}{2} \int_0^1 \int_0^1 \int_0^1 \frac{x^n (1 - x)^n y^n (1 -
   y)^n z^n (1 - z)^n}{(1 - (1 - x y) z)^{n + 1}} \id x \id y \id z
\end{equation*}
and shows that
\begin{equation*}
  J_n = A (n) \zeta (3) + B (n),
\end{equation*}
where
\begin{equation*}
  A (n) = \sum_{k = 0}^n \binom{n}{k}^2 \binom{n + k}{k}^2
\end{equation*}
are the Ap\'ery numbers and $B (n)$ are certain rational numbers (satisfying the
same three-term recurrence as the Ap\'ery numbers). For instance,
\begin{align*}
  J_0 & = \frac{1}{2} \int_0^1 \int_0^1 \int_0^1 \frac{1}{1 - (1 - x y) z}
  \id x \id y \id z = \zeta (3),\\
  J_1 & = \frac{1}{2} \int_0^1 \int_0^1 \int_0^1 \frac{x (1 - x) y (1 - y) z
  (1 - z)}{(1 - (1 - x y) z)^2} \id x \id y \id z = 5 \zeta (3) - 6.
\end{align*}
Beukers \cite{beukers1} shows that $ d_n^3 J_n = a_n \zeta (3) + b_n $, with $ d_n = \operatorname{lcm}(1,2,\ldots,n) $, are integer linear combinations
of $\zeta(3)$ and $1$ (that is, $a_n, b_n \in \mathbb{Z}$).
He then deduces the irrationality of $\zeta(3)$ from the bounds $ 0 < |a_n \zeta (3) + b_n| < (4/5)^n $, which hold for large enough $n$.

Recently, Brown \cite{brown-apery} introduced cellular integrals, generalizing Beukers' integrals $J_n$,
and showed that these cellular integrals are special linear forms in multiple zeta values, which reproduce (and vastly generalize)
many of the known constructions related to irrationality questions for zeta values.
We close by wondering whether the triple integral \eqref{eq:Z3} can be similarly embedded in an infinite family of linear forms in $\zeta(3)$ and $1$.


\bigskip 


\small
\begin{thebibliography}{10}

\bibitem{andrews-1999a}
G.~E. Andrews, R.~Askey, and R.~Roy.
\newblock {\em Special {F}unctions}, volume~71 of {\em Encyclopedia of Mathematics and its Applications}.
\newblock Cambridge University Press, New York, 1999.

\bibitem{berndtI}
B.~Berndt.
\newblock {\em Ramanujan's {N}otebooks, {P}art {I}}.
\newblock Springer-Verlag, New York, 1985.

\bibitem{berndtIV}
B.~Berndt.
\newblock {\em Ramanujan's {N}otebooks, {P}art {IV}}.
\newblock Springer-Verlag, New York, 1994.

\bibitem{beukers1}
F.~Beukers.
\newblock A note on the irrationality of $\zeta (2)$ and $\zeta (3)$.
\newblock {\em Bull. London Math. Soc.}, 11:268--272, 1979.
  
\bibitem{bb-feynman}
C.~Bogner and F.~Brown.
\newblock Feynman integrals and iterated integrals on moduli spaces of curves of genus zero.
\newblock {\em Comm. Number Theory Phys.}, 9:189--238, 2015.

\bibitem{b3l}
J.~M.~Borwein, D.~M.~Bradley, D.~J.~Broadhurst, and P.~Lisonek.
\newblock Special Values of Multiple Polylogarithms.
\newblock {\em Transactions of the American Mathematical Society}, 353:907--941, 2001.

\bibitem{borweinj-2006a}
J.~M.~Borwein and D.~Bradley. 
\newblock Thirty two Goldbach variations. 
\newblock {\em Int. J. Number Theory}, 2:65--103, 2006.

\bibitem{brown-apery}
F.~Brown.
\newblock Irrationality proofs for zeta values, moduli spaces and dinner parties.
\newblock {\em Mosc. J. Comb. Number Theory}, 6(2--3):102--165, 2016.

\bibitem{gradshteyn-2015a}
I.~S. Gradshteyn and I.~M. Ryzhik.
\newblock {\em Table of {I}ntegrals, {S}eries, and {P}roducts}.
\newblock Edited by D. Zwillinger and V. Moll. Academic Press, New York, 8th edition, 2015.

\bibitem{hoffman-refs}
M.~Hoffman.
\newblock {\em References on multiple zeta values and Euler sums}.
\newblock Available at: \url{https://www.usna.edu/Users/math/meh/biblio.html}

\bibitem{kontsevich-2000a}
M.~Kontsevich and D.~Zagier.
\newblock Periods.
\newblock In {\em Mathematics unlimited$-2001$ and beyond}, pages 771--808. Springer-Berlin, 2000.

\bibitem{lewin-1981a}
L.~Lewin.
\newblock {\em Dilogarithms and {A}ssociated {F}unctions}.
\newblock Elsevier, North Holland, 2nd. edition, 1981.

\bibitem{lewin-1981b}
L.~Lewin.
\newblock {\em Polylogarithms and {A}ssociated {F}unctions}.
\newblock North Holland, New York, Oxford, 1981.

\bibitem{hpl}
D.~Maitre.
\newblock {HPL}, a {M}athematica implementation of the harmonic polylogarithms.
\newblock {\em Comput.~{P}hys.~{C}omm.}, 174:222--240, 2006.

\bibitem{olver-2010a}
F.~W.~J. Olver, D.~W. Lozier, R.~F. Boisvert, and C.~W. Clark, editors.
\newblock {\em {NIST} {H}andbook of {M}athematical {F}unctions}.
\newblock Cambridge {U}niversity {P}ress, 2010.

\bibitem{panzer-hyperint}
E.~Panzer.
\newblock Algorithms for the symbolic integration of hyperlogarithms with applications to {F}eynman integrals.
\newblock {\em Comput.~{P}hys.~{C}omm.}, 188:148--166, 2015.

\bibitem{paule-rr}
P.~Paule.
\newblock Short and easy computer proofs of the {R}ogers-{R}amanujan identities and of identities of similar type.
\newblock {\em Electron.~J.~Combin.}, 1:\#R10 1--9, 1994.

\bibitem{rm-hpl}
E.~Remiddi and J.~A.~M. {V}ermaseren.
\newblock Harmonic polylogarithms.
\newblock {\em Internat.~J.~Modern~Phys.~A}, 15:725--754, 2000.

\end{thebibliography}
\end{document}